\documentclass{amsart}
\usepackage{amsfonts}
\usepackage{amssymb}
\usepackage{amsmath}

\newtheorem{theorem}{Theorem}[section]
\theoremstyle{plain}
\newtheorem{acknowledgement}{Acknowledgement}

\newtheorem{corollary}[theorem]{Corollary}

\newtheorem{remark}[theorem]{Remark}

\numberwithin{equation}{section}
\input{tcilatex}

\begin{document}
\title[Univalence Criteria ]{Univalence Criteria and Quasiconformal
Extensions}
\author{Murat \c{C}A\u{G}LAR}
\address{Department of Mathematics, Faculty of Science, Ataturk University,
Erzurum, 25240, Turkey.}
\email{mcaglar25@gmail.com}
\author{Halit ORHAN}
\address{Department of Mathematics, Faculty of Science, Ataturk University,
Erzurum, 25240, Turkey.}
\email{orhanhalit607@gmail.com}

\begin{abstract}
In the present paper, we obtain a more general conditions for univalence of
analytic functions in the open unit disk $\mathcal{U}.$ Also, we obtain a
refinement to a quasiconformal extension criterion of the main result.
\end{abstract}

\maketitle

\section{Introduction}

Let $\mathcal{A}$ be the class of analytic functions $f$ in the open unit
disk $\mathcal{U=}\left\{ z\in 
\mathbb{C}
:\left\vert z\right\vert <1\right\} $ with $f(0)=f^{\prime }(0)-1=0.$ We
denote by $\mathcal{U}_{r}$ the disk $\left\{ z\in {\mathbb{C}}:\left\vert
z\right\vert <r\right\} ,$ where $0<r\leq 1$, by $\mathcal{U}=\mathcal{U}%
_{1} $ the open unit disk of the complex plane and by $I$ the interval $%
[0,\infty )$.

Most important and known univalence criteria for analytic functions defined
in the open unit disk were obtained by Becker \cite{Bec}, Nehari \cite{Neh}
and Ozaki-Nunokawa \cite{Oz-Nu}. Some extensions of these three criteria
were given by (see \cite{Kan}, \cite{Mia-Wes}, \cite{Ove}, \cite{Radu1}-\cite%
{Radu5} and \cite{Tudor}). During the time, a lot of univalence criteria
were obtained by different authors (see also \cite{CaOr}, \cite{Deniz1}-\cite%
{Gol}).

In the present investigation we use the method of subordination chains to
obtain some sufficient conditions for the univalence of an analytic
function. Also, by using Becker's method, we obtain a refinement to a
quasiconformal extension criterion of the main result.

\section{Preliminaries}

Before proving our main theorem we need a brief summary of the method of
Loewner chains and quasiconformal extensions.

A function $L(z,t):\mathcal{U}\times \lbrack 0,\infty )\rightarrow \mathbb{C}
$ is said to be \textit{subordination chain} \textit{(or} \textit{Loewner
chain)} if:

\begin{itemize}
\item[(i)] $L(z,t)$ is analytic and univalent in $\mathcal{U}$ for all $%
t\geq 0$.

\item[(ii)] $L(z,t)\prec L(z,s)$ for all $0\leq t\leq s<\infty $, where the
symbol $"\prec "$ stands for subordination.
\end{itemize}

In proving our results, we will need the following theorem due to Ch.
Pommerenke \cite{Pom}.

\begin{theorem}
\label{t1}\textbf{\ }\textit{Let }$L(z,t)=a_{1}(t)z+a_{2}(t)z^{2}+...,%
\;a_{1}(t)\neq 0$\textit{\ be analytic in }$\mathcal{U}_{r}$\textit{\ for
all }$t\in I,$\textit{\ locally absolutely continuous in }$I,$\textit{\ and
locally uniform with respect to }$\mathcal{U}_{r}$\textit{$.$ For almost all 
}$t\in I,$\textit{\ suppose that}%
\begin{equation}
z\frac{\partial L(z,t)}{\partial z}=p(z,t)\frac{\partial L(z,t)}{\partial t},%
\text{ }\forall z\in \mathcal{U}_{r}  \label{1.1}
\end{equation}%
\textit{where }$p(z,t)$\textit{\ is analytic in }$\mathcal{U}$\textit{\ and
satisfies the condition }$\Re p(z,t)>0$\textit{\ for all }$z\in \mathcal{U}%
,\;t\in I.$\textit{\ If }$\left\vert a_{1}(t)\right\vert \rightarrow \infty $%
\textit{\ for }$t\rightarrow \infty $\textit{\ and }$\{L(z,t)\diagup
a_{1}(t)\}$\textit{\ forms a normal family in }$\mathcal{U}_{r},$\textit{\
then for each }$t\in I,$\textit{\ the function }$L(z,t)$\textit{\ has an
analytic and univalent extension to the whole disk }$\mathcal{U}.$
\end{theorem}

Let $k$ be constant in $[0,1).$ Then a homeomorphism $f$ of $G\subset 
\mathbb{C}$ is said to be $k-$\textit{quasiconformal, }if $\partial _{z}f$
and $\partial _{\overline{z}}f$ in the distributional sense are locally
integrable on $G$ and fulfill the inequality $\left\vert \partial _{%
\overline{z}}f\right\vert \leq k\left\vert \partial _{z}f\right\vert $
almost everywhere in $G.$ If we do not need to specify $k,$ we will simply
call that $f$ is \textit{quasiconformal.}

The method of constructing quasiconformal extension criteria is based on the
following result due to Becker (see \cite{Bec}, \cite{Be2} and also \cite%
{Be3}).

\begin{theorem}
\label{t2*}Suppose that $L(z,t)$ is a Loewner chain. Consider%
\begin{equation*}
w(z,t)=\frac{p(z,t)-1}{p(z,t)+1},\;z\in \mathcal{U},\;t\geq 0
\end{equation*}
where $p(z,t)$ is given in (\ref{1.1}). If%
\begin{equation*}
\left\vert w(z,t)\right\vert \leq k,\;0\leq k<1
\end{equation*}%
for all $z\in \mathcal{U}$ and $t\geq 0,$ then $L(z,t)$ admits a continuous
extension to $\overline{\mathcal{U}}$ for each $t\geq 0$ and the function $%
F(z,\overline{z})$ defined by%
\begin{equation*}
F(z,\overline{z})=\left\{ 
\begin{array}{c}
\mathcal{L}(z,0),\text{ \ \ \ \ \ \ \ if \ }\left\vert z\right\vert <1 \\ 
\mathcal{L}(\frac{z}{\left\vert z\right\vert },\log \left\vert z\right\vert
),\text{ \ if \ }\left\vert z\right\vert \geq 1%
\end{array}%
\right.
\end{equation*}%
is a $k-$quasiconformal extension of $L(z,0)$ to ${\mathbb{C}}.$
\end{theorem}

Examples of quasiconformal extension criteria can be found in \cite{Ah}, 
\cite{AnHi}, \cite{Bet}, \cite{Kr}, \cite{Pf} and more recently in \cite%
{Deniz}, \cite{Ho1}-\cite{Ho3}, \cite{Radu6}.

\section{Main Results}

Making use of Theorem \ref{t1} we can prove now, our main results.

\begin{theorem}
\label{t2} Consider $f\in \mathcal{A}$ and $g$ be an analytic function in $%
\mathcal{U},\;g(z)=1+b_{1}z+....$ Let $\alpha ,\beta ,A$ and $B$ complex
numbers such that $\Re (\alpha )>\frac{1}{2},$ $A+B\neq 0,\;\left\vert
A-B\right\vert <2,\;\left\vert A\right\vert \leq 1$ and $\left\vert
B\right\vert \leq 1.$ If the inequalities%
\begin{equation}
\left\vert \frac{1}{\alpha }\left( \frac{f^{\prime }(z)}{g(z)-\beta }%
-1\right) \right\vert <\frac{\left\vert A+B\right\vert }{2-\left\vert
A-B\right\vert }  \label{eq4}
\end{equation}%
and%
\begin{equation}
\left\vert \left( \frac{f^{\prime }(z)}{g(z)-\beta }-1\right) \left\vert
z\right\vert ^{2}+\left( 1-\left\vert z\right\vert ^{2}\right) \left[ \left( 
\frac{1-\alpha }{\alpha }\right) \frac{zf^{\prime }(z)}{f(z)}+\frac{%
zg^{\prime }(z)}{g(z)-\beta }\right] -\frac{\left( \overline{A}-\overline{B}%
\right) \left( A+B\right) }{4-\left\vert A-B\right\vert ^{2}}\right\vert
\leq \frac{2\left\vert A+B\right\vert }{4-\left\vert A-B\right\vert ^{2}}
\label{eq5}
\end{equation}%
are satisfied for all $z\in \mathcal{U},$ then the function $f$ is univalent
in $\mathcal{U}$.
\end{theorem}

\begin{proof}
We prove that there exists a real number $r\in \left( 0,1\right] $ such that
the function $L:\mathcal{U}_{r}\times I\rightarrow {\mathbb{C}},$ defined
formally by%
\begin{equation}
L(z,t)=f^{1-\alpha }(e^{-t}z)\left[ f(e^{-t}z)+\left( e^{t}-e^{-t}\right)
z\left( g(e^{-t}z)-\beta \right) \right] ^{\alpha }  \label{3.3}
\end{equation}%
is analytic in $\mathcal{U}_{r}$ for all $t\in I.$

Since $f(z)\neq 0$ for all $z\in \mathcal{U}\diagdown \{0\},$ the function%
\begin{equation}
\varphi _{1}(z,t)=\frac{\left( e^{t}-e^{-t}\right) z\left( g(e^{-t}z)-\beta
\right) }{f(e^{-t}z)}  \label{3.4}
\end{equation}%
is analytic in $\mathcal{U}$.

It follows from 
\begin{equation}
\varphi _{2}(z,t)=1+\frac{\left( e^{t}-e^{-t}\right) z\left(
g(e^{-t}z)-\beta \right) }{f(e^{-t}z)}  \label{3.5}
\end{equation}%
that there exist a $r_{1},$ $0<r_{1}<r$ such that $\varphi _{2}$ is analytic
in $\mathcal{U}_{r_{1}}$ and $\varphi _{2}(0,t)=\left( 1-\beta \right)
e^{2t}+\beta ,\;\varphi _{2}(z,t)\neq 0$ for all $z\in \mathcal{U}_{r_{1}},$ 
$t\in I.$ Therefore, we choose an analytic branch in $\mathcal{U}_{r_{1}}$
of the function 
\begin{equation}
\varphi _{3}(z,t)=\left[ \varphi _{2}(z,t)\right] ^{\alpha }.  \label{3.6}
\end{equation}%
From these considerations it follows that the function 
\begin{eqnarray*}
L(z,t) &=&f^{1-\alpha }(e^{-t}z)\left[ f(e^{-t}z)+\left( e^{t}-e^{-t}\right)
z\left( g(e^{-t}z)-\beta \right) \right] ^{\alpha } \\
&=&f(e^{-t}z)\varphi _{3}(z,t)=a_{1}(t)z+....
\end{eqnarray*}%
is an analytic function in $\mathcal{U}_{r_{1}}$ for all $t\in I$ $.$

After simple calculation we have%
\begin{equation}
a_{1}(t)=e^{(2\alpha -1)t}\left[ \beta e^{-2t}+1-\beta \right] ^{\alpha }
\label{3.7}
\end{equation}%
for which we consider the uniform branch equal to $1$ at the origin. Because 
$\Re (\alpha )>\frac{1}{2},$ we have that%
\begin{equation*}
\lim_{t\rightarrow \infty }|a_{1}(t)|=\infty .
\end{equation*}%
Moreover, $a_{1}(t)\neq 0$ for all $t\in I.$

From the analyticity of $L(z,t)$ in $\mathcal{U}_{r_{1}},$ it follows that
there exists a number $r_{2},$ $0<r_{2}<r_{1},$ and a constant $K=K(r_{2})$
such that%
\begin{equation*}
\left\vert \frac{L(z,t)}{a_{1}(t)}\right\vert <K,\text{ }\forall z\in 
\mathcal{U}_{r_{2}},\;t\in I.
\end{equation*}%
Then, by Montel's Theorem, $\left\{ \frac{L(z,t)}{a_{1}(t)}\right\} _{t\in
I} $ is a normal family in $\mathcal{U}_{r_{2}}.$ From the analyticity of $%
\frac{\partial L(z,t)}{\partial t},$ we obtain that for all fixed numbers $%
T>0$ and $r_{3},\;0<r_{3}<r_{2},$ there exists a constant $K_{1}>0$ (that
depends on $T$ and $r_{3}$) such that%
\begin{equation*}
\left\vert \frac{\partial L(z,t)}{\partial t}\right\vert <K_{1},\text{ }%
\forall z\in \mathcal{U}_{r_{3}},\;t\in \left[ 0,T\right] .
\end{equation*}%
Therefore, the function $L(z,t)$ is locally absolutely continuous in $I,$
locally uniform with respect to $\mathcal{U}_{r_{3}}.$

Let $p:\mathcal{U}_{r}\times I\rightarrow {\mathbb{C}}$ be the analytic
function in $\mathcal{U}_{r},\;0<r<r_{3},$ for all $t\in I,$ defined by%
\begin{equation*}
p(z,t)={\frac{\partial L(z,t)}{\partial t}\diagup z\frac{\partial L(z,t)}{%
\partial z}}.
\end{equation*}%
If the function%
\begin{equation}
w(z,t)=\frac{p(z,t)-1}{A+Bp(z,t)}=\frac{\frac{\partial L(z,t)}{\partial t}-%
\frac{z\partial L(z,t)}{\partial z}}{A\frac{z\partial L(z,t)}{\partial z}+B%
\frac{\partial L(z,t)}{\partial t}}  \label{3.8}
\end{equation}%
is analytic in $\mathcal{U}\times I$ and $\left\vert w(z,t)\right\vert <1,$
for all $z\in \mathcal{U}\;$and$\mathrm{\;}t\in I,$ then $p(z,t)$ has an
analytic extension with positive real part in $\mathcal{U},$ for all $t\in
I. $ From equality (\ref{3.8}) we have%
\begin{equation}
w(z,t)=\frac{-2\phi (z,t)}{\left( A-B\right) \phi (z,t)+A+B}  \label{3.9}
\end{equation}%
for $z\in \mathcal{U}$ and $t\in I,$ where%
\begin{equation}
\phi (z,t)=\left( \frac{1}{\alpha }\frac{f^{\prime }(e^{-t}z)}{%
g(e^{-t}z)-\beta }-1\right) e^{-2t}+\left( 1-e^{-2t}\right) \left[ \left( 
\frac{1-\alpha }{\alpha }\right) \frac{e^{-t}zf^{\prime }(e^{-t}z)}{%
f(e^{-t}z)}+\frac{e^{-t}zg^{\prime }(e^{-t}z)}{g(e^{-t}z)-\beta }\right]
\label{3.10}
\end{equation}

From (\ref{eq4}), (\ref{3.9}), (\ref{3.10}) and $\Re (\alpha )>\frac{1}{2}$
we have%
\begin{equation}
\left\vert w(z,0)\right\vert =\left\vert \frac{1}{\alpha }\left( \frac{%
f^{\prime }(z)}{g(z)-\beta }-1\right) \right\vert <\frac{\left\vert
A+B\right\vert }{2-\left\vert A-B\right\vert }  \label{3.13}
\end{equation}%
and%
\begin{equation}
\left\vert w(0,t)\right\vert =\left\vert \left( \frac{1}{\alpha \left(
1-\beta \right) }-1\right) e^{-2t}\right\vert <\frac{\left\vert
A+B\right\vert }{2-\left\vert A-B\right\vert }.  \label{3.14}
\end{equation}%
where $A+B\neq 0,\;\left\vert A-B\right\vert <2,\;\left\vert A\right\vert
\leq 1$ and $\left\vert B\right\vert \leq 1.$

Since $\left\vert e^{-t}z\right\vert \leq \left\vert e^{-t}\right\vert
=e^{-t}<1$ for all $z\in \overline{\mathcal{U}}=\left\{ z\in {\mathbb{C}}%
:\;\left\vert z\right\vert \leq 1\right\} $ and $t>0,$ we find that $w(z,t)$
is an analytic function in $\overline{\mathcal{U}}.$ Using the maximum
modulus principle it follows that for all $z\in \mathcal{U}-\{0\}$ and each $%
t>0$ arbitrarily fixed there exists $\theta =\theta (t)\in {\mathbb{R}}$
such that%
\begin{equation}
\left\vert w(z,t)\right\vert <\max_{\left\vert z\right\vert =1}\left\vert
w(z,t)\right\vert =\left\vert w(e^{i\theta },t)\right\vert ,  \label{3.15}
\end{equation}%
for all $z\in \mathcal{U}\;$and$\mathrm{\;}t\in I.$

Denote $u=e^{-t}e^{i\theta }.$ Then $\left\vert u\right\vert =e^{-t}$ and
from (\ref{3.9}) we have%
\begin{equation*}
\left\vert w(e^{i\theta },t)\right\vert =\left\vert \frac{2\phi (e^{i\theta
},t)}{\left( A-B\right) \phi (e^{i\theta },t)+A+B}\right\vert
\end{equation*}%
where%
\begin{equation*}
\phi (e^{i\theta },t)=\left( \frac{1}{\alpha }\frac{f^{\prime }(u)}{%
g(u)-\beta }-1\right) \left\vert u\right\vert ^{2}+\left( 1-\left\vert
u\right\vert ^{2}\right) \left[ \left( \frac{1-\alpha }{\alpha }\right) 
\frac{uf^{\prime }(u)}{f(u)}+\frac{ug^{\prime }(u)}{g(u)-\beta }\right]
\end{equation*}%
Because $u\in \mathcal{U},$ the inequality (\ref{eq5}) implies that%
\begin{equation*}
\left\vert w(e^{i\theta },t)\right\vert \leq 1,
\end{equation*}%
for all $z\in \mathcal{U}\;$and$\mathrm{\;}t\in I.$ Therefore $\left\vert
w(z,t)\right\vert <1$ for all $z\in \mathcal{U}\;$and$\mathrm{\;}t\in I.$

Since all the conditions of Theorem \ref{t1} are satisfied, we obtain that
the function $L(z,t)$ has an analytic and univalent extension to the whole
unit disk $\mathcal{U},$ for all $t\in I.$ For $t=0$ we have $L(z,0)=f(z),$
for $z\in \mathcal{U}$ and therefore the function $f$ is analytic and
univalent in $\mathcal{U}.$
\end{proof}

If we take $A=B$ in Theorem \ref{t2}, we get the following univalence
criterion.

\begin{corollary}
\label{11}Consider $f\in \mathcal{A}$ and $g$ be an analytic function in $%
\mathcal{U},\;g(z)=1+b_{1}z+....$ Let $\alpha ,\beta \;$and $A$ are complex
numbers such that $\Re (\alpha )>\frac{1}{2},$ $A\neq 0,\;\left\vert
A\right\vert \leq 1.$ If the inequalities%
\begin{equation}
\left\vert \frac{1}{\alpha }\left( \frac{f^{\prime }(z)}{g(z)-\beta }%
-1\right) \right\vert <\left\vert A\right\vert  \label{eq7}
\end{equation}%
and%
\begin{equation}
\left\vert \left( \frac{f^{\prime }(z)}{g(z)-\beta }-1\right) \left\vert
z\right\vert ^{2}+\left( 1-\left\vert z\right\vert ^{2}\right) \left[ \left( 
\frac{1-\alpha }{\alpha }\right) \frac{zf^{\prime }(z)}{f(z)}+\frac{%
zg^{\prime }(z)}{g(z)-\beta }\right] \right\vert \leq \left\vert A\right\vert
\label{eq8}
\end{equation}%
are satisfied for all $z\in \mathcal{U},$ then the function $f$ is univalent
in $\mathcal{U}$.
\end{corollary}

If we choose $\alpha =1$ in Theorem \ref{t2}, we obtain the following
univalence criterion.

\begin{corollary}
\label{c1}Consider $f\in \mathcal{A}$ and $g$ be an analytic function in $%
\mathcal{U},\;g(z)=1+b_{1}z+....$ Let $\beta ,A$ and $B$ complex numbers
such that $A+B\neq 0,\;\left\vert A-B\right\vert <2,\;\left\vert
A\right\vert \leq 1$ and $\left\vert B\right\vert \leq 1.$ If the
inequalities%
\begin{equation}
\left\vert \frac{f^{\prime }(z)}{g(z)-\beta }-1\right\vert <\frac{\left\vert
A+B\right\vert }{2-\left\vert A-B\right\vert }  \label{eq9}
\end{equation}%
and%
\begin{equation}
\left\vert \left( \frac{f^{\prime }(z)}{g(z)-\beta }-1\right) \left\vert
z\right\vert ^{2}+\left( 1-\left\vert z\right\vert ^{2}\right) \frac{%
zg^{\prime }(z)}{g(z)-\beta }-\frac{\left( \overline{A}-\overline{B}\right)
\left( A+B\right) }{4-\left\vert A-B\right\vert ^{2}}\right\vert \leq \frac{%
2\left\vert A+B\right\vert }{4-\left\vert A-B\right\vert ^{2}}  \label{eq10}
\end{equation}%
are satisfied for all $z\in \mathcal{U},$ then the function $f$ is univalent
in $\mathcal{U}$.
\end{corollary}

\begin{corollary}
\label{c2}Consider $f\in \mathcal{A}$ and $g$ be an analytic function in $%
\mathcal{U},\;g(z)=1+b_{1}z+....$ Let $A$ and $B$ complex numbers such that $%
A+B\neq 0,\;\left\vert A-B\right\vert <2,\;\left\vert A\right\vert \leq 1$
and $\left\vert B\right\vert \leq 1.$ If the inequalities%
\begin{equation}
\left\vert \frac{f^{\prime }(z)}{g(z)}-1\right\vert <\frac{\left\vert
A+B\right\vert }{2-\left\vert A-B\right\vert }  \label{eq11}
\end{equation}%
and%
\begin{equation}
\left\vert \left( \frac{f^{\prime }(z)}{g(z)}-1\right) \left\vert
z\right\vert ^{2}+\left( 1-\left\vert z\right\vert ^{2}\right) \frac{%
zg^{\prime }(z)}{g(z)}-\frac{\left( \overline{A}-\overline{B}\right) \left(
A+B\right) }{4-\left\vert A-B\right\vert ^{2}}\right\vert \leq \frac{%
2\left\vert A+B\right\vert }{4-\left\vert A-B\right\vert ^{2}}  \label{eq12}
\end{equation}%
are satisfied for all $z\in \mathcal{U},$ then the function $f$ is univalent
in $\mathcal{U}$.
\end{corollary}

\begin{proof}
It results from Corollary \ref{c1} with $\alpha =1$ and $\beta =0.$
\end{proof}

For $g(z)=f^{\prime }(z)$ in Corollary \ref{c2}, we have the following
univalence criterion.

\begin{corollary}
\label{c3}Consider $f\in \mathcal{A}.$ Let $A$ and $B$ complex numbers such
that $A+B\neq 0,\;\left\vert A-B\right\vert <2,\;\left\vert A\right\vert
\leq 1$ and $\left\vert B\right\vert \leq 1.$ If the inequality%
\begin{equation}
\left\vert \left( 1-\left\vert z\right\vert ^{2}\right) \frac{zf^{\prime
\prime }(z)}{f^{\prime }(z)}-\frac{\left( \overline{A}-\overline{B}\right)
\left( A+B\right) }{4-\left\vert A-B\right\vert ^{2}}\right\vert \leq \frac{%
2\left\vert A+B\right\vert }{4-\left\vert A-B\right\vert ^{2}}  \label{eq13}
\end{equation}%
is satisfied for all $z\in \mathcal{U},$ then the function $f$ is univalent
in $\mathcal{U}$.
\end{corollary}

\begin{corollary}
\label{c4}Consider $f\in \mathcal{A}.$ Let $\beta ,A$ and $B$ complex
numbers such that $A+B\neq 0,\;\left\vert A-B\right\vert <2,\;\left\vert
A\right\vert \leq 1$ and $\left\vert B\right\vert \leq 1.$ If the
inequalities%
\begin{equation}
\left\vert \frac{\beta }{f^{\prime }(z)-\beta }\right\vert <\frac{\left\vert
A+B\right\vert }{2-\left\vert A-B\right\vert }  \label{eq14}
\end{equation}%
and%
\begin{equation}
\left\vert \frac{\beta \left\vert z\right\vert ^{2}+\left( 1-\left\vert
z\right\vert ^{2}\right) zf^{\prime \prime }(z)}{f^{\prime }(z)-\beta }-%
\frac{\left( \overline{A}-\overline{B}\right) \left( A+B\right) }{%
4-\left\vert A-B\right\vert ^{2}}\right\vert \leq \frac{2\left\vert
A+B\right\vert }{4-\left\vert A-B\right\vert ^{2}}  \label{eq15}
\end{equation}%
are satisfied for all $z\in \mathcal{U},$ then the function $f$ is univalent
in $\mathcal{U}$.
\end{corollary}

\begin{proof}
It results from Corollary \ref{c1} with $g(z)=f^{\prime }(z).$
\end{proof}

\begin{corollary}
\label{c5}Consider $\beta <0$ in Corollary \ref{c4}. By elementary
calculation we obtain that the inequality (\ref{eq14}) for $A=B=1$ is
equivalent to%
\begin{equation*}
\Re f^{\prime }(z)>\frac{1}{2\beta }\left\vert f^{\prime }(z)\right\vert
^{2},\;z\in \mathcal{U}.
\end{equation*}%
If in the last inequality we let $\beta \rightarrow -\infty $ we obtain that%
\begin{equation*}
\Re f^{\prime }(z)>0.
\end{equation*}

Since (\ref{eq15}) for $A=B=1$ and $\beta \rightarrow -\infty $ it follows
from Corollary \ref{c4} that the function $f$ is univalent in $\mathcal{U}$.
Therefore, we can conclude that the univalence criterion due to
Alexander-Noshiro-Warshawski \cite{Alex}, \cite{Nos}, \cite{Wars} is a limit
case of Corollary \ref{c4}.
\end{corollary}

\begin{remark}
\label{r1}Some particular cases of Theorem \ref{t2} are the following:

(i) When $\alpha =1,$ $\beta =0,$ $A=B=1$ and $g(z)=f^{\prime }(z)$
inequality (\ref{eq5}) becomes%
\begin{equation}
\left( 1-\left\vert z\right\vert ^{2}\right) \left\vert \frac{zf^{\prime
\prime }(z)}{f^{\prime }(z)}\right\vert \leq 1,\text{ }z\in \mathcal{U}
\label{eq6}
\end{equation}%
which is Becker's condition of univalence \cite{Bec}.

(ii) A result due to N. N. Pascu \cite{Pas} is obtained when $\alpha =1,$ $%
A=B=1$ and $g(z)=f^{\prime }(z).$
\end{remark}

\begin{remark}
\label{r2}It is worth to notice that the condition (\ref{eq5}) assures the
univalence of an analytic function in more general case than that of
condition (\ref{eq6}).
\end{remark}

\begin{remark}
\label{r3}If we put $g(z)=\frac{f(z)}{z}$ into (\ref{eq12}), we have%
\begin{equation*}
\left\vert \frac{zf^{\prime }(z)}{f(z)}-1\right\vert \leq \frac{\left\vert
A+B\right\vert }{2-\left\vert A-B\right\vert },\;z\in \mathcal{U}
\end{equation*}%
the class of functions starlike with respect to origin.
\end{remark}

\section{Quasiconformal Extension}

In this section we will obtain the univalence condition given in Theorem \ref%
{t2} to a quasiconformal extension criterion.

\begin{theorem}
\label{t3}Consider $f\in \mathcal{A},$ $g$ be an analytic function in $%
\mathcal{U},\;g(z)=1+b_{1}z+...$ and $k\in \left[ 0,1\right) .$ Let $\alpha
,\beta ,A$ and $B$ complex numbers such that $\Re (\alpha )>\frac{1}{2},$ $%
A+B\neq 0,\;k\left\vert A-B\right\vert <2,\;\left\vert A\right\vert \leq 1$
and $\left\vert B\right\vert \leq 1.$ If the inequalities%
\begin{equation}
\left\vert \frac{1}{\alpha }\left( \frac{f^{\prime }(z)}{g(z)-\beta }%
-1\right) \right\vert <\frac{k\left\vert A+B\right\vert }{2-k\left\vert
A-B\right\vert }  \label{eq7*}
\end{equation}%
and%
\begin{equation}
\left\vert \left( \frac{f^{\prime }(z)}{g(z)-\beta }-1\right) \left\vert
z\right\vert ^{2}+\left( 1-\left\vert z\right\vert ^{2}\right) \left[ \left( 
\frac{1-\alpha }{\alpha }\right) \frac{zf^{\prime }(z)}{f(z)}+\frac{%
zg^{\prime }(z)}{g(z)-\beta }\right] -\frac{k^{2}\left( \overline{A}-%
\overline{B}\right) \left( A+B\right) }{4-k^{2}\left\vert A-B\right\vert ^{2}%
}\right\vert \leq \frac{2k\left\vert A+B\right\vert }{4-k^{2}\left\vert
A-B\right\vert ^{2}}  \label{eq8*}
\end{equation}%
are satisfied for all $z\in \mathcal{U},$ then the function $f$ has a $k-$%
quasiconformal extension to $%
\mathbb{C}
.$
\end{theorem}

\begin{proof}
In the proof of Theorem \ref{t2} has been proved that the function $L(z,t)$
given by (\ref{3.3}) is a subordination chain in $\mathcal{U}.$ Applying
Theorem \ref{t2*} to the function $w(z,t)$ given by (\ref{3.9}), we obtain
that the assumption%
\begin{equation}
\left\vert w(z,t)\right\vert =\left\vert \frac{-2\phi (z,t)}{\left(
A-B\right) \phi (z,t)+A+B}\right\vert \leq k,\;z\in \mathcal{U},\;t\geq
0,\;k\in \left[ 0,1\right)  \label{eq9*}
\end{equation}%
where $\phi (z,t)$ is defined by (\ref{3.10}).

Lenghty but elementary calculation shows that the last inequality (\ref{eq9*}%
) is equivalent to%
\begin{equation*}
\left\vert \left( \frac{1}{\alpha }\frac{f^{\prime }(e^{-t}z)}{%
g(e^{-t}z)-\beta }-1\right) e^{-2t}+\left( 1-e^{-2t}\right) \left[ \left( 
\frac{1-\alpha }{\alpha }\right) \frac{e^{-t}zf^{\prime }(e^{-t}z)}{%
f(e^{-t}z)}+\frac{e^{-t}zg^{\prime }(e^{-t}z)}{g(e^{-t}z)-\beta }\right]
\right.
\end{equation*}%
\begin{equation}
\left. -\frac{k^{2}\left( \overline{A}-\overline{B}\right) \left( A+B\right) 
}{4-k^{2}\left\vert A-B\right\vert ^{2}}\right\vert \leq \frac{2k\left\vert
A+B\right\vert }{4-k^{2}\left\vert A-B\right\vert ^{2}}.  \label{eq10*}
\end{equation}%
The inequality (\ref{eq10*}) implies $k-$quasiconformal extensibility of $f.$

The proof is complete.
\end{proof}

If we choose $\alpha =1,\beta =0,g=f^{\prime }$ in Theorem \ref{t3}, we
obtain following corollary.

\begin{corollary}
\label{c6}Consider $f\in \mathcal{A}$ and $k\in \left[ 0,1\right) .$ Let $A$
and $B$ complex numbers such that $A+B\neq 0,\;k\left\vert A-B\right\vert
<2,\;\left\vert A\right\vert \leq 1$ and $\left\vert B\right\vert \leq 1.$
If the inequality%
\begin{equation}
\left\vert \left( 1-\left\vert z\right\vert ^{2}\right) \left( \frac{%
zf^{\prime \prime }(z)}{f^{\prime }(z)}\right) -\frac{k^{2}\left( \overline{A%
}-\overline{B}\right) \left( A+B\right) }{4-k^{2}\left\vert A-B\right\vert
^{2}}\right\vert \leq \frac{2k\left\vert A+B\right\vert }{4-k^{2}\left\vert
A-B\right\vert ^{2}}  \label{eq11*}
\end{equation}%
is satisfied for all $z\in \mathcal{U},$ then the function $f$ has a $k-$%
quasiconformal extension to $%
\mathbb{C}
.$
\end{corollary}

For $A=B=1$ in Corollary \ref{c6}, we have result of Becker \cite{Bec}.

\begin{corollary}
\label{c7}Consider $f\in \mathcal{A}$ and $k\in \left[ 0,1\right) .$ If the
inequality%
\begin{equation}
\left( 1-\left\vert z\right\vert ^{2}\right) \left\vert \frac{zf^{\prime
\prime }(z)}{f^{\prime }(z)}\right\vert \leq k  \label{eq12*}
\end{equation}%
is satisfied for all $z\in \mathcal{U},$ then the function $f$ has a $k-$%
quasiconformal extension to $%
\mathbb{C}
.$
\end{corollary}

\begin{acknowledgement}
The present investigation was supported by Atat\"{u}rk University Rectorship
under \textquotedblright The Scientific and Research Project of Atat\"{u}rk
University\textquotedblright , Project No: 2012/173.
\end{acknowledgement}

\end{document}